\documentclass[12pt]{article}
\usepackage{amsmath,amscd,amsthm,amssymb}

\def\R{\mathbb{R}}
\def\i{\infty}

\def\R+{\mathbb{R}_{+}}
\def\Rn{\mathbb{R}^n}

\newtheorem{theorem}{Theorem}[section]
\newtheorem{lemma}{Lemma}[section]
\newtheorem{corollary}{Corollary}[section]

\theoremstyle{definition}

\def\Lpwloc{L_{p,w}^{\rm loc}(\Rn)}
\def\Lpwloc1{L_{p,w}^{\rm loc}(\Omega)}
\def\L1loc{L_{1}^{\rm loc}(\Rn)}
\def\L1wloc{L_{1,w}^{\rm loc}(\Rn)}
\def\L1wloc1{L_{1,w}^{\rm loc}(\Omega)}

\def\dual{\,^{^{\complement}}\!}

\newtheorem{remark}{Remark}[section]

\numberwithin{equation}{section}

\newcommand{\ess}{\mathop{\rm ess \; sup}\limits}
\newcommand{\es}{\mathop{\rm ess \; inf}\limits}

\usepackage{color}

\numberwithin{equation}{section}

\begin{document}

\begin{center}
\LARGE The Dirichlet problem for the uniformly higher-order elliptic equations in generalized weighted Sobolev-Morrey spaces
\end{center}

\

\centerline{\large V.S. Guliyev$^{a,b,}$
\footnote{Corresponding author : vagif@guliyev.com (V.S. Guliyev); tgadjiev@mail.az (T.S. Gadjiev); serbetci@ankara.edu.tr (A. Serbetci)}, T.S. Gadjiev$^{a}$,
A. Serbetci$^{c,d}$}

\

\centerline{$^{a}$\it Institute of Mathematics and Mechanics, Baku, Azerbaijan}

\centerline{$^{b}$\it Department of Mathematics, Dumlupinar University, Kutahya, Turkey}

\centerline{$^{c}$\it Department of Mathematics, \c{C}ankiri Karatekin University, \c{C}ankiri, Turkey}
\centerline{$^{d}$\it Department of Mathematics, Ankara University, Ankara, Turkey}

\

\begin{abstract}
A priori estimates for the weak solutions the Dirichlet problem for the uniformly higher-order elliptic equations
in a smooth bounded domain $\Omega\subset \Rn$ in generalized weighted Sobolev-Morrey spaces are obtained.
%Weight from in the Muckenhoupt class $A_{p}$.
\end{abstract}

\



\section{Introduction}
$~~~$ Recall that the classical {\it Morrey spaces} $L_{p,\lambda}$ were introduced  in  \cite{Morr1938} in order to study the local behavior of the solutions of elliptic  systems.
Moreover, various Morrey spaces are defined in the process of study. Guliyev, Mizuhara and Nakai \cite{GulDoc, Miz1991, Nakai1994} introduce generalized Morrey spaces $M_{p,\varphi}$. Komori and Shirai \cite{KomShir} define weighted Morrey spaces $L_{p,\kappa}(w)$; Guliyev \cite{GulEMJ2012} give a concept of the generalized weighted Morrey spaces $M_{p,\varphi}(w)$ which could be viewed as extension of both $M_{p,\varphi}$ and $L_{p,\kappa}(w)$, study the boundedness of the classical operators and their commutators in spaces $M_{p,\varphi}(w)$ was studied (see, also \cite{GulOmAzJM,HamzTrAMEA}).

 Let ${W^{2m}_{p}(\Omega)}$ be the standard notation for Sobolev spaces. In \cite{AgmDouNir} for the solutions of uniformly elliptic equations in a smooth domain $\Omega$ the following a priori estimate
\begin{equation}\label{eq-1.1}
 \|u \|_{W^{2m}_{p}(\Omega)}\leq C \|f \|_{L_{p}(\Omega)}
\end{equation}
were obtained.
In \cite{Muc} on a bounded domain $\Omega$ with smooth boundary $\partial \Omega$ for the Laplace equation with weight $w(x)$ belonging to the Muckenhoupt class $A_{p}$ (see \cite{ChaWhe})  was proved the following a priori estimate
\begin{equation*}
 \|u \|_{W^{2}_{p}(\Omega,w)}\leq C \|f \|_{L_{p}(\Omega,w)}.
\end{equation*}

Weighted estimates for a wide class of singular integral operators has been obtained for weights in the class of Muckenhoupt $A_{p}$. Therefore, it is a natural question whether  analogous weighted a priori estimates can be proved for the derivativies of solutions elliptic equations.
In \cite{DurSanTos} the previous results of \cite{ChaWhe} (also \cite{az, af, as}) for powers of the Laplacian operator with homogeneous Dirichlet boundary conditions were extended to weighted Sobolev spaces, i.e., it is proved that
\begin{equation*}
 \|u \|_{W^{2m}_{p}(\Omega,w)}\leq C \|f \|_{L_{p}(\Omega,w)},
\end{equation*}
for $\omega\in A_{p}$, where the constant $C$ depends on $\Omega$, $m$, $n$ and $w$.

In \cite{GGG_EJQTDE2017}, Guliyev, Gadjiev and Galandarova study the boundedness of the sublinear operators generated by Calderon-Zygmund operators in local generalized Morrey spaces. By using these results they prove the solvability of the Dirichlet boundary value problem for a polyharmonic equation in modified local generalized Sobolev-Morrey spaces and obtain a priori estimates for the solutions of the Dirichlet boundary value problems for the uniformly elliptic equations in modified local generalized Sobolev-Morrey spaces defined on bounded smooth domains.

Main purpose of this paper is to generalize Calderon-Zygmund type estimates of weak solution in generalized weighted Sobolev-Morrey spaces.
These estimates play an important role in regularity theory with H\"{o}lder estimates, studies have examined for classical $L_p$ estimates
or their generalizations. We apply these estimates to study the regularity of the solution of Dirichlet problem for linear elliptic partial differential equations (see \cite{ErOmMurProcIMM2017, GulOm2015, GulMurOmSoft2016,GulOmSoftProcIMM}). The presented results are generalization of previous works
\cite{DallSw2004, DurSanTos, vx, GGG_EJQTDE2019,GulAhmOmSoft2018}.

The paper is organized as follows. In Section 2, we give some definitions and auxiliary results. We also give some estimates of Green function and Poisson kernel.
In Section 3,  we study the boundedness of the maximal singular integral operators in generalized weighted Morrey spaces. In Section 4,  we
give regularity estimates and solvability of higher-order elliptic equations in generalized weighted Sobolev-Morrey spaces ${W^{m}M_{p,\varphi}(\Omega,w)}$. In Section 5, we prove the solvability of a uniformly elliptic boundary value problem in ${W^{m}M_{p,\varphi}(\Omega,w)}$.

We denote by $c$, $C, c_1, c_2$ etc., the various absolute  positive constants, which may have different values even in the same line.

\section{Preliminaries}

$~~~$ Let $\Rn$ be $n$-dimensional Euclidean space, $\Omega \subset \Rn$
be an open set. We denote by $\chi_E$ the characteristic function of a set $E\subseteq \Rn$ and $B(x,r)=\{y \in \Rn :|x-y| < r\}$.

Consider the homogeneous problem
\begin{equation}\label{eq-2.1}
\left\{
\begin{array}{l}
(-\Delta)^{m}u = f \,\,\,\mbox{in} \,\,\,\, \Omega,
\\
\left(\frac{\partial}{\partial \upsilon}\right)^{j}u = 0 \,\,\, \mbox{in}\,\,\,\, \partial\Omega, \,\,\, 0\leq j\leq m-1,
\end{array}\right.
\end{equation}
where $\frac{\partial}{\partial \upsilon}$ is the normal derivative, in a bounded domain $\Omega$ with smooth boundary $\partial\Omega$. The solution of \eqref{eq-2.1}
is given by
\begin{equation}\label{eq-2.2}
u(x)=\int\limits_{\Omega}G_{m}(x,y)f(y)dy,
\end{equation}
where $G_{m}(x,y)$ is Green function of the operator which can be written as
\begin{equation}\label{eq-2.3}
G_{m}(x,y)=\Gamma(x-y)+h(x,y)
\end{equation}
in $\Omega$, here $\Gamma(x-y)$ is a fundamental solution and $h(x,y)$ satisfies
\begin{equation*}
(-\Delta_{x})^{m}h(x,y) = 0, x\in \Omega,
$$$$
\left(\frac{\partial}{\partial \nu}\right)^{j}h(x,y)=-\left(\frac{\partial}{\partial \nu}\right)^{j}\Gamma(x-y), \, x\in\partial\Omega, ~ 0\leq j\leq m-1
\end{equation*}
for each fixed $y\in \Omega$. Then
\begin{equation*}
h(x,y)=-\sum\limits^{m-1}_{j=0}\int\limits_{\partial\Omega}
K_{j}(y,P)\left(\frac{\partial}{\partial \nu}\right)^{j}\Gamma(P-x)ds,
\end{equation*}
where $K_{j}(y,P)$ are Poisson kernels, $ds$ denotes the surface measure on $\partial\Omega$. We have the following known estimates of Green function $G_{m}(x,y)$  and Poisson kernels $K_{j}(x,y)$:
\begin{equation}\label{eq-2.4}
|D^{\alpha}_{x}G_{m}(x,y)|\leq C, \,\,\, \mbox{for} \,\,\, |\alpha|<2m-n,
\end{equation}
\begin{equation}\label{eq-2.5}
|D^{\alpha}G_{m}(x,y)|\leq C \, \log\left(\frac{2d}{|x-y|}\right), \,\,\, \mbox{for} \,\,\, |\alpha|=2m-n,
\end{equation}
\begin{equation}\label{eq-2.6}
|D^{\alpha}G_{m}(x,y)|\leq C \, |x-y|^{2m-n-|\alpha|}, \,\,\, \mbox{for} \,\,\, |\alpha|>2m-n,
\end{equation}
\begin{equation}\label{eq-2.7}
|D^{\alpha}G_{m}(x,y)|\leq \frac{C}{|x-y|^{n}} \,  \min\left\{1,\frac{d(y)}{|x-y|}\right\}^{m}, \,\,\, \mbox{for} \,\,\, |\alpha|=2m,
\end{equation}
\begin{equation}\label{eq-2.8}
|K_{j}(x,y)|\leq \frac{C \, d(x)}{|x-y|^{n-j+m-1}}, \,\,\, \mbox{for} \,\,\, 0\leq j\leq m-1,
\end{equation}
where $d(x)=dist (x,\partial\Omega)$ and $d=diam(\Omega)$ (see \cite{DurSanTos,vx}).

The Hardy-Littlewood maximal operator $M$ is defined by
$$
Mf(x)=\sup\limits_{t>0}|B(x,t)|^{-1}\int_{B(x,t)}|f(y)|dy.
$$

Now we give some known pointwise estimates.
\begin{lemma}\label{lemm-2.1} $\cite{GGG_EJQTDE2017}$ Let $u(x)$ be the solution of the problem \eqref{eq-2.1} and $|\alpha|\leq 2m-n$. Then there exists a constant $C$ depending on $n,m$ and $\Omega$ for all $x\in \Omega$ such that
\begin{equation*}
|D^{\alpha}u(x)|\leq C \, Mf(x).
\end{equation*}
%where $Mf(x)$ is the usual Hardy-Littlewood maximal function of $f$.
\end{lemma}
\begin{lemma}\label{lemm-2.2}  $\cite{DurSanTos}$
Let $f$, $g$ be measurable functions on $\Omega$, $|\alpha|=2m$ and $D=\left\{(x,y)\in \Omega\times \Omega : |x-y|>d(x)\right\}$. Then there exists a constant $C$ depending on $n,m$ and $\Omega$ such that
\begin{equation*}
\int_{D}|D^{\alpha}G_{m}(x,y)f(y)g(x)|dxdy\leq C \,  \left(\int_{D}Mf(y)|g(x)|dx+\int_{D}Mg(y)|f(y)|dy\right).
\end{equation*}
\end{lemma}
In order to see how to estimate $D^{\alpha}_{x}h(x,y)$ in $\Omega\backslash D$, we consider separately the functions $h(x,y)$ and $\Gamma(x,y)$ involved in $G_{m}(x,y)$.
\begin{lemma}\label{lemm-2.3}   $\cite{DurSanTos}$
If $|\alpha|>2m-n+1$, then there exists a constant $C$ such that
\begin{equation}\label{eq-2.9}
|D^{\alpha}_{x}h(x,y)|\leq C \, d^{2m-n-|\alpha|}(x) ~~~ \mbox{for} ~~ |x-y|\leq d(x).
\end{equation}
\end{lemma}
Let $K$ be a Calderon-Zygmund singular integral operator, briefly a Calderon-Zygmund operator is a linear operator bounded from $L_{2}(\Rn)$  to $L_{2}(\Rn)$ taking all infinitely continuously differentiable functions $f$ with compact support to functions in $L_{1}^{\rm loc}(\Rn)$, represented for such functions by
$$
Kf(x)=\int_{\Rn}K(x,y)f(y)dy,
$$
here $K(x,y)$ is a continuous function which satisfies the standard estimates.

It follows from the previous lemmas that for each $x\in \Omega$ and $|\alpha|>2m-n+1$ we have $D^{\alpha}_{x}h(x,y)$ is bounded uniformly in a neighborhood of $x$ and so
\begin{equation}\label{eq-2.10}
D^{\alpha}_{x}\int_{\Omega}h(x,y)f(y)dy=\int_{\Omega}D^{\alpha}_{x}h(x,y)f(y)dy.
\end{equation}
On the other hand, although $D^{\alpha}_{x}\Gamma(x,y)$ is a singular kernel for $|\alpha|=2m$, taking $\beta$ such that $|\beta|=2m-1$, we have that
\begin{equation}\label{eq-2.11}
D^{\alpha}_{x}\int_{\Omega}D^{\beta}_{x}\Gamma(x-y)f(y)dy=Kf(x)+a(x)f(x),
\end{equation}
where $a(x)$ is a bounded function and $K$ is a Calderon-Zygmund operator given by
\begin{equation*}
  Kf(x)=\lim\limits_{\varepsilon\rightarrow 0}K_{\varepsilon}f(x) \,\,\,\mbox{with} \,\,\, K_{\varepsilon}f(x)=\int_{\Rn \setminus B(x,\varepsilon)} D^{\alpha}_{x}\Gamma(x-y)f(y)dy.
\end{equation*}
We will also make use of the maximal singular operator $K^{\ast}f(x)=\sup\limits_{\varepsilon>0}|K_{\varepsilon}f(x)|$.
Here and in what follows we consider $f$ defined in $\Rn$ extending the original $f$ by zero.

\begin{lemma}\label{lemm-2.4}   $\cite{DurSanTos}$
Let $g(x)$ be a measurable function on $\Omega$ and $|\alpha|=2m$. Then there exists a constant $C$ depending  only on $n,m$ and $\Omega$ such that
\begin{align*}
\int_{\Omega}|D^{\alpha}u(x) \, g(x)|dx & \leq C \, \Big(\int_{\Omega} K^{\ast}f(x) \, |g(x)|dx + \int_{\Omega} Mf(x) \, |g(x)|dx
\\
& + \int_{\Omega} Mg(x) \, |f(x)| dx + \int_{\Omega} |f(x)| \, |g(x)|dx \Big).
\end{align*}
\end{lemma}

%We define the generalized weighted Morrey spaces $M_{p,\varphi}(\Rn,w)$.

%\begin{definition}\label{def-2.1}
	Let $1\leq p<\infty$, $\varphi$ be a positive measurable function on $\Rn\times (0,\infty)$ and $w$ be nonnegative measurable function on $\Rn$. We denote by $M_{p,\varphi}(\Rn,w)$ the generalized weighted Morrey spaces, the space of all functions $f\in L_{p,w}^{\rm loc}(\Rn)$ with finite norm
	$$
	\|f\|_{M_{p,\varphi}(w)}=\sup\limits_{x\in \Rn,r>0}\varphi^{-1}(x,r) \,w (B(x,r))^{-\frac{1}{p}}\|f\|_{L_{p,w}(B(x,r))},
	$$
	where $L_{p,w}(B(x,r))$ denotes the weighted $L_{p}$-space of measurable functions $f$ for which
	$$
	\|f\|_{L_{p,w}(B(x,r))}\equiv \|f_{\chi_{B(x,r)}}\|_{L_{p,w}(\Rn)}=\left(\int_{B(x,r)}|f(y)|^{p} w(y)dy \right)^{\frac{1}{p}}.
	$$
	Furthermore, by $WM_{p,\varphi}(w) \equiv WM_{p,\varphi}(\Rn,w)$ we denote the weak generalized weighted Morrey space of all functions $f\in WL_{p,w}^{\rm loc}(\Rn)$ for which
	$$
	\|f\|_{WM_{p,\varphi}(w)}=\sup\limits_{x\in \Rn,r>0}\varphi^{-1}(x,r) \,w (B(x,r))^{-\frac{1}{p}}\|f\|_{WL_{p,w}(B(x,r))}<\infty,
	$$
	where $WL_{p,w}(B(x,r))$ denotes the weak $L_{p,w}$-space of measurable functions $f$ for which
	$$
	\|f\|_{WL_{p,w}(B(x,r))}\equiv \|f_{\chi_{B(x,r)}}\|_{WL_{p,w}(\Rn)}=\sup\limits_{t>0}\left(\int_{y\in B(x,r):|f(y)|>t} w(y)dy \right)^{\frac{1}{p}}.
	$$
%\end{definition}
\begin{remark}\label{rem-2.1}
	\begin{enumerate}
		\item If $w \equiv 1$, then $M_{p,\varphi}(1)=M_{p,\varphi}(\Rn)$ is the generalized Morrey space.
		\item If $\varphi(x,r)=w(B(x,r))^{\frac{k-1}{p}}$, then $M_{p,\varphi}(w)=L_{p,k}(w)$ is the weighted Morrey space.
		\item If $\varphi(x,r)=w_1(B(x,r))^{\frac{k}{p}} \, w(B(x,r))^{-\frac{1}{p}}$, then $M_{p,\varphi}(w)=L_{p,k}(w,w_1)$ is the two weighted Morrey space.
		\item If $w=1$ and $\varphi(x,r)=r^{\frac{\lambda-n}{p}}$ with $0<\lambda<n$, then $M_{p,\varphi}(w)=L_{p,\lambda}(\Rn)$ is the classical Morrey space and $WM_{p,\varphi}(w)=WL_{p,\lambda}(\Rn)$ is the weak Morrey space.
		\item If $\varphi(x,r)=w(B(x,r))^{-\frac{1}{p}}$, then $M_{p,\varphi}(w)=L_{p,w}(\Rn)$ is the weighted Lebesgue space.
	\end{enumerate}
\end{remark}

For any bounded domain $\Omega$ we define $M_{p,\varphi}(\Omega,w)$ taking $f \in L_{p,w} (\Omega)$ and $\Omega(x,r)$ instead of $B(x,r)$ in the norm above and $\Omega(x,r)=\Omega\cap B(x,r)$.
The generalized weighted Sobolev-Morrey spaces $W^{m}M_{p,\varphi}(\Omega,w)$ consist of all weighted Sobolev functions $u\in W^{m}_{p}(\Omega,w)$  with distributional derivatives $D^{s}u\in M_{p,\varphi}(\Omega,w)$, $0\leq |s|\leq m$, endowed with the norm
$$
\|u\|_{W^{m}M_{p,\varphi}(\Omega,w)}=\sum\limits_{0\leq |s|\leq m}\|D^{s}u\|_{M_{p,\varphi}(\Omega,w)}.
$$
The space $W^{m}M_{p,\varphi}(\Omega,w)\cap C^{\infty}_{0}(\Omega)=\overset{\circ}{W^{m}} M_{p,\varphi}(\Omega,w)$.

We recall the definition of $A_{p}(\Omega)$ class for $1\le p<\infty$.
A non-negative locally integrable function $w(x)$ belongs to $A_{p}(\Omega)$ if there exists a constant $C$ such that
\begin{equation*}
  \left(\frac{1}{|B|}\int_{B}w(x)dx\right) \left(\frac{1}{|B|}\int_{B}w^{-\frac{1}{p-1}}(x)dx\right)^{p-1}\leq C,
\end{equation*}
for all ball $B\subset \Omega$.

\begin{theorem}\label{theor-2.1} $\cite{GulEMJ2012}$
	Let $1\leq p<\infty$, $w\in A_{p}(\Rn)$ and the pair $(\varphi_{1},\varphi_{2})$ satisfy the condition
\begin{equation}\label{eq-2.13}
\int^{\infty}_{r}\frac{\es_{t<s<\infty}\varphi_{1}(x,s) w(B(x,s))^{\frac{1}{p}}}{w(B(x,s))^{\frac{1}{p}}}\frac{dt}{t}\leq C\varphi_{2}(x,r),
\end{equation}
where $C$ does not depend on $x$ and $r$.  Then the operators $M$ and $K$ are bounded from $M_{p,\varphi_{1}}(w)$ to $M_{p,\varphi_{2}}(w)$ for $p>1$ and from $M_{1,\varphi_{1}}(w)$ to $WM_{1,\varphi_{2}}(w)$.
\end{theorem}

\section{Boundedness of the maximal singular operators in $M_{p,\varphi}(\Omega,w)$}

In this section we prove the boundedness of the maximal singular operators $K^{\ast}$ in generalized weighted Morrey spaces $M_{p,\varphi}(\Omega,w)$.
We are going to use the following statement on the boundedness of the weighted Hardy operator
$$
H^*_{w} g(r):= \int_r^d\, g(t) w(t)\, dt, \qquad 0<r<d\,.
$$
where $w$ is a fixed function non-negativeand measurable on $(0,d)$.

The following theorem was proved in \cite{GulAJM2013}.
\begin{theorem} \label{Hardy} %$(\cite{GulJMS2013, GulAJM2013})$
	Suppose that $v_1, v_2,$ and $w$ are weights on $(0,d)$.  Then the inequality
	\begin{equation}\label{eqH1}
	\ess_{0<r<d} v_2(r)H^*_{w} g(r) \leq C \ess_{0<r<d} v_1(r) g(r)
	\end{equation}
	holds with some $C>0$ for all nonnegative and nondecreasing $g$ on $(0,d)$ if and only if
	\begin{equation}\label{eqH2}
	B:=\ess_{0<r<d} v_2(r)\int_r^d\,  \frac{w(t)}{\ess_{t<s<d}v_1(s)}\, dt<\infty
	\end{equation}
	and $C=B$ is the best constant in  \eqref{eqH1}.
\end{theorem}
\begin{remark}
In \eqref{eqH1} and \eqref{eqH2} it is assumed that $\frac{1}{\infty}=0$ and $0 \cdot \infty = 0$.
\end{remark}

In the following lemma we give local estimates for the maximal singular integral.
\begin{lemma}\label{MarOk1}
	Let $1\le p<\infty$ and $w\in A_{p}(\Omega)$.
	Then for $p>1$ the inequality
	\begin{equation*}\label{eq3.5.}
	\|K^{\ast} f\|_{L_{p,w}(\Omega(x_0,r))} \lesssim w(\Omega(x_0,r))^{\frac{1}{p}} \int_{2r}^{d} \|f\|_{L_{p,w}(\Omega(x_0,t))} \, w(\Omega(x_0,t))^{-\frac{1}{p}} \, \frac{dt}{t}
	\end{equation*}
	holds for any ball $\Omega(x_0,r)$, and for all $f\in\Lpwloc1$.
	
	Moreover, for $p=1$ the inequality
	\begin{equation*}\label{eq3.5.WX}
	\|K^{\ast} f\|_{WL_{1,w}(\Omega(x_0,r))} \lesssim w(\Omega(x_0,r)) \int_{2r}^{d} \|f\|_{L_{1,w}(\Omega(x_0,t))} \, w(\Omega(x_0,t))^{-1} \, \frac{dt}{t}
	\end{equation*}
	holds for any ball $\Omega(x_0,r)$, and for all $f\in\L1wloc1$.
\end{lemma}
\begin{proof}
	For arbitrary $x_0\in\Rn,$ set $\Omega_0=\Omega(x_0,r)$ for the ball centered at $x_0$ and of radius $r$, $2\Omega_0=\Omega(x_0,2r).$ We represent $f$ as
	\begin{equation}\label{denk4.2}
	f=f_1+f_2,~~f_1(y)=f(y)\chi_{2\Omega_0}(y),~~f_2(y)=f(y)\chi_{\dual(2B)}(y),~~r>0
	\end{equation}
	and have
	$$
	\|K^{\ast} f\|_{L_{p,w}(\Omega_0)}\le \|K^{\ast} f_1\|_{L_{p,w}(\Omega_0)}+
	\|K^{\ast} f_2\|_{L_{p,w}(\Omega_0)}.
	$$
	
	Since $f_1\in L_{p,w}(\Omega)$, $K^{\ast} f_1 \in L_{p,w}(\Omega)$ and from the boundedness of $K^{\ast} $ in $L_{p,w}(\Omega)$ for $w\in A_{p}(\Omega)$ (see, for example, \cite[Corollary 7.13]{Duoandik}) it follows that
	\begin{align*}
	\|K^{\ast} f_1\|_{L_{p,w}(\Omega)}
	&\leq \|K^{\ast} f_1\|_{L_{p,w}(\Omega)}
	\lesssim  \, \|f_1\|_{L_{p,w}(\Omega)}
	\approx  \, \|f\|_{L_{p,w}(2\Omega_0)}.
	\end{align*}
	
	It is clear that $x\in \Omega_0$, $y\in \dual (2\Omega_0)$ implies $\frac{1}{2}|x_0-y|\le|x-y|\le \frac{3}{2}|x_0-y|$. Then by the Minkowski inequality and conditions on $\Omega$, we get
	\begin{align*}
	K^{\ast} f_2(x)  & \lesssim \int_{\dual(2\Omega_0)}\frac{|f(y)|}{|x_0-y|^n}dy.
	\end{align*}
	
	By Fubini's theorem we have
	\begin{align*}
	&\int_{\dual (2\Omega_0)} \frac{|f(y)|}{|x_0-y|^{n}} dy \approx
	\int_{\dual (2\Omega_0)} |f(y)| \, \int_{|x-y|}^{\i}\frac{dt}{t^{n+1}}dy
	\\
	& = \int_{2r}^{d}\int_{2r\leq |x_0-y|< t} |f(y)|dy \, \frac{dt}{t^{n+1}}
	\lesssim \int_{2r}^{d}\int_{\Omega(x_0,t)} |f(y)|dy\frac{dt}{t^{n+1}}\ .
	\end{align*}
	By applying H\"older's inequality for $w\in A_{p}(\Omega)$, we get
	\begin{align}\label{sal00}
	&\int_{\dual (2\Omega_0)}\frac{|f(y)|}{|x_0-y|^{n}}dy \lesssim \int_{2r}^{d} \|f\|_{L_1(\Omega(x_0,t))}\, \frac{dt}{t^{n+1}} \notag
	\\
	& \lesssim \int_{2r}^{d}\|f\|_{L_{p,w}(\Omega(x_0,t))} \, \|w^{-1/p}\|_{L_{p'}(\Omega(x_0,t))}  \, \frac{dt}{t^{n+1}} \notag
	\\
	& \lesssim  \, \int_{2r}^{d}\|f\|_{L_{p,w}(\Omega(x,t))} \, w(\Omega(x_0,t))^{-\frac{1}{p}} \, |\Omega(x_0,t)|   \,
	\frac{dt}{t^{n+1}} \notag
	\\
	&\lesssim  \, \int_{2r}^{\i}\|f\|_{L_{p,w}(\Omega(x_0,t))}w(\Omega(x_0,t))^{-\frac{1}{p}}\,\frac{dt}{t}.
	\end{align}
	
	Moreover, for all $p\in (1,\i)$ the inequality
	\begin{align*} \label{ves2}
	\|K^{\ast} f_2\|_{L_{p,w}(\Omega_0)}&\lesssim
	w(\Omega_0)^{\frac{1}{p}} \, \int_{2r}^{d}\|f\|_{L_{p,w}(\Omega(x_0,t))}\, w(\Omega(x_0,t))^{-\frac{1}{p}}\,\frac{dt}{t}.
	\end{align*}
	is valid. Thus
\begin{align*}
\|K^{\ast} f\|_{L_{p,w}(\Omega_0)}&\lesssim  \|f\|_{L_{p,w}(2\Omega_0)}
+w(B)^{\frac{1}{p}}\int_{2r}^{\i}\|f\|_{L_{p,w}(\Omega(x_0,t))}\, w(\Omega(x_0,t))^{-\frac{1}{p}}\,\frac{dt}{t}.
\end{align*}
On the other hand,
\begin{align*}
&\|f\|_{L_{p,w}(2\Omega_0)}
\lesssim  |\Omega_0| \, \int_{2r}^{d} \|f\|_{L_{p,w}(\Omega(x_0,t))}\,\frac{dt}{t^{n+1}}
\\
& \lesssim  w(\Omega_0)^{\frac{1}{p}} \, \|w^{-1/p}\|_{L_{p'}(\Omega_0)} \, \int_{2r}^{d} \|f\|_{L_{p,w}(\Omega(x_0,t))} \, \frac{dt}{t^{n+1}}
\end{align*}
\begin{align*}
& \lesssim w(\Omega_0)^{\frac{1}{p}} \, \int_{2r}^{d} \|f\|_{L_{p,w}(\Omega(x_0,t))} \, \|w^{-1/p}\|_{L_{p'}(\Omega(x_0,t))} \, \frac{dt}{t^{n+1}}
\\
& \lesssim  \, w(\Omega_0)^{\frac{1}{p}} \, \int_{2r}^{d} \|f\|_{L_{p,w}(\Omega(x_0,t))} \, w(\Omega(x_0,t))^{-\frac{1}{p}} \, \frac{dt}{t}.
\end{align*}
Thus
\begin{align*}
\|K^{\ast} f\|_{L_{p,w}(\Omega_0)}&\lesssim  w(\Omega_0)^{\frac{1}{p}}
\int_{2r}^{d} \|f\|_{L_{p,w}(\Omega(x_0,t))} \,w(\Omega(x_0,t))^{-\frac{1}{p}}\,\frac{dt}{t}.
\end{align*}
	
	Let $p=1$. From the weak $(1,1)$ boundedness of $K^{\ast}$  it follows
	that
	\begin{align*}
	\|K^{\ast} f_1\|_{WL_{1,w}(\Omega_0)}
	& \leq \|K^{\ast} f_1\|_{WL_{1,w}(\Omega)}
	\lesssim  \|f_1\|_{L_{1,w}(\Omega)}= \|f\|_{L_{1,w}(2\Omega_0)}
	\\
	 &\lesssim w(\Omega_0)  \int_{2r}^{d}\|f\|_{L_{1,w}(\Omega(x,t))} w(\Omega(x,t))^{-1}\frac{dt}{t}.
	\end{align*}
	
Thus we complete the proof of Lemma \ref{MarOk1}.
\end{proof}

\begin{theorem}\label{Kuzuf1}
Let $1\le p<\i$, $w\in A_{p}(\Omega)$ and the pair $(\varphi_1,\varphi_2)$ satisfy the condition \eqref{eq-2.13}.
Then the operator $K^{\ast} $ is bounded from $M_{p,\varphi_1}(\Omega,w)$ to $M_{p,\varphi_2}(w)$ for $p>1$ and from $M_{1,\varphi_1}(w)$ to $WM_{1,\varphi_2}(w)$ for $p=1$.
\end{theorem}
\begin{proof}
By Lemma \ref{MarOk1} and Theorem \ref{Hardy} with $\nu_1(r)=\varphi_{1}(x,r)^{-1}w(\Omega(x,t))^{-\frac{1}{p}}$,
$\nu_{2}(r)=\varphi_{2}(x,r)^{-1}$ and $w(r)=w(\Omega(x,t))^{-\frac{1}{p}}$ we have for $p>1$
\begin{align*}
\|K^{\ast} f\|_{M_{p,\varphi_2}(w)} &\lesssim
\sup_{x\in\Rn,\,r>0}\varphi_2(x,r)^{-1}
\int_r^{\i}\|f\|_{L_{p,w}(\Omega(x,t))}\,
w(\Omega(x,t))^{-\frac{1}{p}}\,\frac{dt}{t}
\\
&\lesssim \sup_{x\in\Rn, r>0} \varphi_1(x,r)^{-1}\, w(\Omega(x,r))^{-\frac{1}{p}}\,\|f\|_{L_{p,w}\Omega(x,r)}
=\|f\|_{M_{p,\varphi_1}(w)}.
\end{align*}
Let $p=1$. By Lemma \ref{MarOk1} and Theorem \ref{Hardy} with $\nu_{2}(r)=\varphi_{2}(x,r)^{-1}$, $\nu_1(r)=\varphi_{1}(x,r)^{-1}w(\Omega(x,t))^{-1}$ and $w(r)=w(\Omega(x,t))^{-1}$ we have
\begin{align*}
\|K^{\ast} f\|_{WM_{1,\varphi_2}(w)} &\lesssim
\sup_{x\in\Rn,\,r>0}\varphi_2(x,r)^{-1}
\int_r^{\i} \|f\|_{L_{p,w}(\Omega(x,t))}\,w(\Omega(x,t))^{-1}\,\frac{dt}{t}
\\
&\lesssim \sup_{x\in\Rn, r>0}\varphi_1(x,r)^{-1}\,w(\Omega(x,r))^{-1}\,\|f\|_{L_{1,w}\Omega(x,r)}
=\|f\|_{M_{1,\varphi_1}(w)}.
\end{align*}
\end{proof}
For $\varphi_{1}(x,r)=\varphi_{2}(x,r)\equiv w(\Omega(x,r))^{\frac{k-1}{p}}$, from Theorem \ref{theor-2.1} we have the following result.
\begin{corollary}\label{cor-2.1}  $\cite{KomShir}$
Let $1\leq p<\infty$, $0\le k<1$ and $w\in A_{p}(\Omega)$. Then the operator
$K^{\ast}$ is bounded on $L_{p,k}(\Omega,w)$ for $p>1$, and bounded from $L_{1,k}(\Omega,w)$ to $WL_{1,k}(\Omega,w)$.
\end{corollary}

\section{Solvability of higher-order elliptic equations  in ${W^{m}M_{p,\varphi}(\Omega,w)}$}

In this section we give an application the boundedness of maximal singular operators in generalized weighted Morrey spaces ${M_{p,\varphi}(\Omega,w)}$ to regularity estimates and solvability of higher-order elliptic equations  in generalized weighted Sobolev-Morrey spaces ${W^{m}M_{p,\varphi}(\Omega,w)}$.

We can now state and prove our main result.
\begin{theorem}\label{theor-3.1}
Let $\Omega\subset \Rn$ be a bounded domain with smooth boundary $\partial\Omega$ and $\varphi$ satisfy the condition
\begin{equation}\label{eq-2.13FG}
\int^{\infty}_{r}\frac{\es_{t<s<\infty}\varphi(x,s) w(B(x,s))^{\frac{1}{p}}}{w(B(x,s))^{\frac{1}{p}}}\frac{dt}{t}\leq C\varphi(x,r),
\end{equation}
where $C$ does not depend on $x$ and $r$. If $w\in A_{p}(\Omega)$, $f\in M_{p,\varphi}(\Omega,w)$ and $u(x)$ a weak solution of \eqref{eq-2.1}, then there exists a constant $C$ depending only $n,m,w$ and $\Omega$ such that
\begin{equation*}
\|u\|_{W^{2m}M_{p,\varphi}(\Omega,w)} \le  C \, \|f\|_{M_{p,\varphi}(\Omega,w)}.
\end{equation*}
\end{theorem}
\begin{proof}
Since $M$ is a bounded operator in $M_{p,\varphi}(\Omega,w)$, by Lemma \ref{lemm-2.1} and Theorem \ref{theor-2.1} it follows that
\begin{equation*}
\sum\limits_{|\alpha|\leq 2m-1}\|D^{\alpha}u\|_{M_{p,\varphi}(\Omega,w)}
\lesssim \|Mf\|_{M_{p,\varphi}(\Omega,w)} \lesssim \, \|f\|_{M_{p,\varphi}(\Omega,w)}.
\end{equation*}
Therefore, it only remains to estimate $\|D^{\alpha}u\|^{p}_{M_{p,\varphi}(\Omega,w)}$ for $|\alpha|= 2m$.

Let $w \in A_{p}(\Omega)$ and $g(x)=(D^{\alpha}u(x))^{p-1}w(x)$. By Lemma \ref{lemm-2.4} we see that
\begin{align}\label{eq-3.1K}
&\|D^{\alpha}u\|_{M_{p,\varphi}(\Omega,w)} =
\sup\limits_{x\in \Omega,r>0}\varphi^{-1}(x,r)\,w (\Omega(x,r))^{-1/p} \, \Big(\int_{\Omega(x,r)} D^{\alpha}u(y) \, |g(y)|dy \Big)^{1/p} \notag
\\
& \le \sup\limits_{x\in \Omega,r>0}\varphi^{-1}(x,r)\,w (\Omega(x,r))^{-1/p} \, \Big(\int_{\Omega(x,r)} K^{\ast}f(y) \, |g(y)|dy + \int_{\Omega(x,r)} Mf(y) \, |g(y)|dy  \notag
\\
&  + \int_{\Omega(x,r)} Mg(y) \, |f(y)|dy + \int_{\Omega(x,r)} |f(y)|\,|g(y)|dy\Big)^{1/p} \le
I + II + III + IV.
\end{align}

By the definition of $g(x)$
\begin{align*}
\int_{\Omega(x,r)} \frac{|g(x)|^{p'}}{w^{\frac{p'}{p}}(x)}dx
& = \int_{\Omega(x,r)}|D^{\alpha}u(x)|^{p} w(x)dx.
\end{align*}

Since $K^{\ast}$ and $M$ are bounded operators in $M_{p,\varphi}(\Omega,w)$, by Corollary \ref{cor-2.1} applying the H\"{o}lder inequality, it follows that
\begin{align}\label{eq-3.2}
& I = \sup\limits_{x\in \Omega,r>0}\varphi^{-1}(x,r) \,w (\Omega(x,r))^{-\frac{1}{p}} \, \Big(\int_{\Omega(x,r)} K^{\ast}f(y) \, |g(y)|dy\Big)^{\frac{1}{p}}  \notag
\\
& \leq \sup\limits_{x\in \Omega,r>0} \varphi^{-1}(x,r) \,w (\Omega(x,r))^{-\frac{1}{p}} \, \Big( \int_{\Omega(x,r)} (K^{\ast}f(y))^{p} w(y)dy \Big)^{\frac{1}{p^2}} \Big(\int_{\Omega(x,r)} \frac{|g(y)|^{p'}}{w^{\frac{p'}{p}}(y)}dy \Big)^{\frac{1}{p'}} \, \notag
\\
& \le \|K^{\ast}f\|_{M_{p,\varphi}(\Omega,w)}^{\frac{1}{p}} \,
\sup\limits_{x\in \Omega,r>0} \, \varphi^{-\frac{1}{p'}}(x,r) \,w (\Omega(x,r))^{-\frac{1}{pp'}} \, \Big(\int_{\Omega(x,r)} \big|D^{\alpha}u(y)\big|^p \, w(y) dy \Big)^{\frac{1}{p'}} \, \notag
\\
& \lesssim \|f\|_{M_{p,\varphi}(\Omega,w)}^{\frac{1}{p}} \,
\|D^{\alpha}u \|_{M_{p,\varphi}(\Omega,w)}^{\frac{1}{p'}},
\end{align}
where $\frac{1}{p}+\frac{1}{p'}=1$.

In the same way, we obtain that
\begin{align}\label{eq-3.3}
& II = \sup\limits_{x\in \Omega,r>0}\varphi^{-1}(x,r) \,w (\Omega(x,r))^{-\frac{1}{p}} \, \Big(\int_{\Omega(x,r)} Mf(y) \, |g(y)|dy\Big)^{1/p}  \notag
\\
& \leq \sup\limits_{x\in \Omega,r>0} \varphi^{-1}(x,r) \,w (\Omega(x,r))^{-\frac{1}{p}} \, \Big( \int_{\Omega(x,r)} (Mf(y))^{p} w(y)dy \Big)^{\frac{1}{p^2}} \Big(\int_{\Omega(x,r)} \frac{|g(y)|^{p'}}{w^{\frac{p'}{p}}(y)}dy \Big)^{\frac{1}{p'}} \, \notag
\\
& \le \|Mf\|_{M_{p,\varphi}(\Omega,w)}^{\frac{1}{p}} \,
\sup\limits_{x\in \Omega,r>0} \, \varphi^{-\frac{1}{p'}}(x,r) \,w (\Omega(x,r))^{-\frac{1}{pp'}} \, \Big(\int_{\Omega(x,r)} \big|D^{\alpha}u(y)\big|^p \, w(y) dy \Big)^{\frac{1}{p'}} \, \notag
\\
& \lesssim \|f\|_{M_{p,\varphi}(\Omega,w)}^{\frac{1}{p}} \,
\|D^{\alpha}u \|_{M_{p,\varphi}(\Omega,w)}^{\frac{1}{p'}},
\end{align}
and
\begin{align}\label{eq-3.4}
& III = \sup\limits_{x\in \Omega,r>0}\varphi^{-1}(x,r) \, w(\Omega(x,r))^{-\frac{1}{p}} \, \Big(\int_{\Omega(x,r)} |f(y)| \, |g(y)|dy\Big)^{\frac{1}{p}}  \notag
\\
& \leq \sup\limits_{x\in \Omega,r>0} \varphi^{-1}(x,r) \, w(\Omega(x,r))^{-\frac{1}{p}} \, \Big( \int_{\Omega(x,r)} |f(y)|^{p} w(y)dy \Big)^{\frac{1}{p^2}} \Big(\int_{\Omega(x,r)} \frac{|g(y)|^{p'}}{w^{\frac{p'}{p}}(y)}dy \Big)^{\frac{1}{p'}} \, \notag
\\
& \le \|f\|_{M_{p,\varphi}(\Omega,w)}^{\frac{1}{p}} \,
\sup\limits_{x\in \Omega,r>0} \, \varphi^{-\frac{1}{p'}}(x,r) \, w(\Omega(x,r))^{-\frac{1}{p'}} \, \Big(\int_{\Omega(x,r)} \big|D^{\alpha}u(y)\big|^p \, w(y) dy \Big)^{\frac{1}{p'}} \, \notag
\\
& \lesssim \|f\|_{M_{p,\varphi}(\Omega,w)}^{\frac{1}{p}} \,
\|D^{\alpha}u \|_{M_{p,\varphi}(\Omega,w)}^{\frac{1}{p'}}.
\end{align}
For the last term in \eqref{eq-3.1K}, taking into account that $w^{-\frac{p'}{p}}\in A_{p'}(\Omega)$, we have that
\begin{align}\label{eq-3.5}
& IV = \sup\limits_{x\in \Omega,r>0}\varphi^{-1}(x,r) \, w(\Omega(x,r))^{-\frac{1}{p}} \, \Big(\int_{\Omega(x,r)} Mg(y) \, |f(y)|dy\Big)^{\frac{1}{p}}  \notag
\\
& \leq \sup\limits_{x\in \Omega,r>0} \varphi^{-1}(x,r) \, w(\Omega(x,r))^{-\frac{1}{p}} \, \Big( \int_{\Omega(x,r)} (f(y))^{p} w(y)dy \Big)^{\frac{1}{p^2}} \Big(\int_{\Omega(x,r)} \frac{(Mg(y))^{p'}}{w^{\frac{p'}{p}}(y)}dy \Big)^{\frac{1}{p'}} \, \notag
\\
& \le \|f\|_{M_{p,\varphi}(\Omega,w)}^{\frac{1}{p}} \,
\sup\limits_{x\in \Omega,r>0} \, \varphi^{-\frac{1}{p'}}(x,r) \, w(\Omega(x,r))^{-\frac{1}{pp'}} \, \Big(\int_{\Omega(x,r)} \big|D^{\alpha}u(y)\big|^p \, w(y) dy \Big)^{\frac{1}{p'}} \, \notag
\\
& \lesssim \|f\|_{M_{p,\varphi}(\Omega,w)}^{\frac{1}{p}} \,
\|D^{\alpha}u \|_{M_{p,\varphi}(\Omega,w)}^{\frac{1}{p'}}.
\end{align}
Then, by \eqref{eq-3.2}, \eqref{eq-3.3}, \eqref{eq-3.4} and  \eqref{eq-3.5} we have
\begin{equation*}
\|D^{\alpha}u\|_{M_{p,\varphi}(\Omega,w)}\leq C \,  \|f\|_{M_{p,\varphi}(\Omega,w)}^{\frac{1}{p}} \,
\|D^{\alpha}u \|_{M_{p,\varphi}(\Omega,w)}^{\frac{1}{p'}}.
\end{equation*}

Then, we obtain
\begin{equation}\label{eq-3.6}
\|D^{\alpha}u\|_{M_{p,\varphi}(\Omega,w)} \lesssim \,  \|f\|_{M_{p,\varphi}(\Omega,w)}
\end{equation}
and the theorem is proved for $u\in W^{2m}M_{p,\varphi}(\Omega,w)$.

It is easy to show that by using classical trace theorems in Sobolev spaces and the definition of $w \in A_{p}$ the weak solution $u$ of \eqref{eq-2.1} belongs to $ W^{2m}M_{p,\varphi}(\Omega,w)$.
\end{proof}

\section{Estimates for any order uniformly elliptic equations.}

Consider a weak solution of Dirichlet problem
\begin{equation}\label{eq-4.1}
\left\{
\begin{array}{l}
Lu = f \,\,\,\mbox{in} \,\,\, \Omega ,
\\
B_{j}u = 0 \,\,\, \mbox{in}\,\,\, \partial\Omega , \,\,\, 0\leq j\leq m-1,
\end{array}\right.
\end{equation}
where $L=\sum\limits_{|\alpha|\leq 2m}a_{\alpha}D^{\alpha}$ - is uniformly elliptic and $B_{j}=\sum\limits_{|\alpha|\leq j}b_{\alpha}D^{\alpha}$ , $0\leq j\leq m-1$ are the boundary operators defined in [1].

There exists a constant $\gamma$ such that
$$
\gamma^{-1}w(x)|\xi|^{2}\leq \sum\limits_{|\alpha|\leq 2m} a_{\alpha}(x)\xi_{\alpha}\xi_{\beta}\leq \gamma w(x)|\xi|^{2},
$$
a.e. $x\in \Omega$, $\forall\xi\in \Rn$ and matrix $a_{\alpha}(x)$ is real symmetrical matrix.

We define $l_{1}>\max\limits_{j}(2m-j)$ and $l_{0}=\max\limits_{j}(2m-j)$. If $a_{\alpha}\in C^{l_{1}+1}({\overline{\Omega}})$, $|\alpha|\leq 2m$, $b_{\alpha} \in C^{l_{1}+1}(\partial \Omega)$, $0\leq j\leq m-1$, and $\partial\Omega \in C^{l_{1}+m+1}$, then we have Green function $G_{m}$ and Poisson kernels $K_{j}$ for $0\leq j\leq m-1$ exist whenever $l_{1}>2(l_{0}+1)$ for $n=2$ and $l_{1}>\frac{3}{2}l_{0}$ for $n\geq 3$.

Moreover, whenever they are defined, Green function and Poisson kernels of the operator $L$ with these boundary conditions satisfy the estimates \eqref{eq-2.4}, \eqref{eq-2.5}, \eqref{eq-2.6},  \eqref{eq-2.7} and  \eqref{eq-2.8} (see [4] and [6]). Then the following result is valid for weak solution of problem \eqref{eq-4.1}.

\begin{theorem}\label{theor-4.1}
Let $\Omega\subset \Rn$ be a bounded domain with smooth boundary $\partial \Omega$ and the coefficients of operators $L$ and $B_{j}$ satisfy the conditions $a_{\alpha}\in C^{l_{1}+1}({\overline{\Omega}})$, $|\alpha|\leq 2m$, $b_{\alpha} \in C^{l_{1}+1}(\partial \Omega)$, $0\leq j\leq m-1$. If $w \in A_{p}(\Omega)$, $f \in M_{p,\varphi}(\Omega,w)$, $\varphi$ satisfies the condition
\eqref{eq-2.13FG} and $u(x)$ is a weak solution of \eqref{eq-4.1}, then there exists a constant $C$ depending only on $n,m,w$ and $\Omega$ such that
\begin{equation}\label{eq-4.2}
\|u\|_{W^{2m}M_{p,\varphi}(\Omega,w)}\leq C \, \|f\|_{M_{p,\varphi}(\Omega,w)}.
\end{equation}
\end{theorem}

The proof Theorem \ref{theor-4.1} is a consequence of the above estimates of the Green function and Lemma \ref{lemm-2.4}. Corollary \ref{cor-2.1} implies that the operators $M$ and $K^{\ast}$ are bounded in $M_{p,\varphi}(\Omega,w)$. Therefore statement of the Theorem \ref{theor-4.1} and estimate \eqref{eq-4.2} are immediately consequence of inequalities in Lemma \ref{lemm-2.4} and Corollary \ref{cor-2.1}.
Thus the theorem is proved.

From Theorems \ref{Kuzuf1} and \ref{theor-4.1}, and estimates in Lemma \ref{lemm-2.4} we get the following corollary.
\begin{corollary}\label{cor-4.1}
Let $\Omega\subset \Rn$ be a bounded domain with smooth boundary $\partial \Omega$ and the coefficients of operators $L$ and $B_{j}$ satisfy the conditions $a_{\alpha}\in C^{l_{1}+1}({\overline{\Omega}})$, $|\alpha|\leq 2m$, $b_{\alpha} \in C^{l_{1}+1}(\partial \Omega)$, $0\leq j\leq m-1$. If $w \in A_{p}(\Omega)$, $f \in M_{p,\varphi_1}(\Omega,w)$, the pair $(\varphi_1,\varphi_2)$ satisfies the condition \eqref{eq-2.13} and $u(x)$ is a weak solution of \eqref{eq-4.1}, then there exists a constant $C$ depending only on $n,m, w$ and $\Omega$ such that
$$
\|u\|_{W^{2m}M_{p,\varphi_{2}}(\Omega,w)}\leq C \|f\|_{M_{p,\varphi_{1}}(\Omega,w)}.
$$
\end{corollary}


\begin{thebibliography}{888}

\bibitem{AgmDouNir}
S. Agmon, A. Douglis and L. Nirenberg, \emph{Estimates near the boundary for solutions of elliptic
partial diferential equations satisfiying general boundary conditions}, Comm. Pure Appl. Math. 12 (1959), 623-727.

\bibitem{ChaWhe}
 S. Chanillo, R. Wheeden, \emph{Harnac's inequality and mean value inequalities for degenerate elliptic equations},
 Com. Pure. Dif. Eq. 11 (1986), 111-134.

\bibitem{ErOmMurProcIMM2017} A. Eroglu, M.N. Omarova, Sh.A. Muradova, \emph{Elliptic equations with measurable coefficients in generalized weighted Morrey spaces}, Proc. Inst. Math. Mech. Natl. Acad. Sci. Azerb. 43 (2017), no. 2, 197-213.

\bibitem{DallSw2004} A. Dall'Acqua, G. Sweers, \textit{Estimates for Green function and Poisson kernels of higher order Dirichlet boundary value problems},
J. Differential Equations 205 (2) (2004), 466-487.

\bibitem{DurSanTos}
 R.G. Duran, M. Sanmartino, M. Toschi, \emph{Weighted a priori estimates for solution $(-\Delta)^{m}n=f$ with homogeneous Dirichlet conditions},
 Anal. Theory Apply. 26 (4) (2010), 339-349.

\bibitem{vx} R.G. Duran, M. Sanmartino, M. Toschi, \textit{Weighted a priori estimates for Poisson equation}, Indiana Univ. Math. J. 57 (2008), 3463-3478.

\bibitem{Duoandik}  J. Duoandikoetxea, Fourier Analysis, Graduate Studies in Mathematics, vol. 29, American Mathematical Society, Providence, RI, 2001.

\bibitem{GGG_EJQTDE2019} T.S. Gadjiev, Sh. Galandarova, V.S. Guliyev, \textit{Regularuty in generalized Morrey spaces of solutions to higher order nondivergence elliptic equations with VMO coefficients}, Electron. J. Qual. Theory Differ. Equ. 2019, Paper No. 55, 17 pp.

\bibitem{az} H.-Ch. Grunau, G. Sweers, \textit{Sharp estimates for iterated Green functions},
Proc. Roy. Soc. Edinburgh (Section A) 132 (1) (2002), 91-120.

\bibitem{af} H.-Ch. Grunau, G. Sweers, \textit{Positivity for equations involving
polyharmonic operators with Dirichlet boundary conditions}, Math. Ann. 307 (4) (1997), 589-626.

\bibitem{as} H.-Ch. Grunau, G. Sweers, \textit{The role of positive boundary data in generalized clamped plate equations}, Z. Angew. Math. Phys. 49 (3) (1998), 420-435


\bibitem{GulDoc} V.S. Guliyev,  \textit{Integral operators on function spaces on the homogeneous groups and on domains in $\mathbb{G}$}, Doctor's degree dissertation, Moscow, Mat. Inst. Steklov, 1994, 1-329. (Russian)

\bibitem{GulEMJ2012} V.S. Guliyev, \emph{Generalized weighted Morrey spaces and higher order commutators of sublinear operators},
Eurasian Math. J. 3 (3) (2012), 33-61.

\bibitem{GulAJM2013}  V.S. Guliyev, \textit{Local generalized Morrey spaces and singular integrals with rough kernel}, Azerb. J. Math. 3 (2) (2013), 79-94.

\bibitem{GulOmAzJM} V.S. Guliyev, M.N. Omarova, \textit{Multilinear singular and fractional integral operators on generalized weighted Morrey spaces}, Azerb. J. Math. 5 (1) (2015), 104-132.

\bibitem{GulOm2015}
V.S. Guliyev, M.N. Omarova, \emph{Parabolic oblique derivative problem with discontinuous coefficients in generalized weighted Morrey spaces},
Open Math. 14 (1) (2016), 49-61.

\bibitem{GulMurOmSoft2016}
V.S. Guliyev, Sh.A. Muradova, M.N. Omarova, L. Softova, \emph{Gradient estimates for parabolic equations in generalized weighted Morrey spaces},
Acta Math. Sin. (Engl. Ser.) 32 (8) (2016), 911-924.

\bibitem{GGG_EJQTDE2017} V.S. Guliyev, T.S. Gadjiev, Sh. Galandarova, \textit{Dirichlet boundary value problems for uniformly elliptic equations in modified local generalized Sobolev-Morrey spaces}, Electron. J. Qual. Theory Differ. Equ. 2017, Paper No. 71, 17 pp.

\bibitem{GulAhmOmSoft2018}
 V.S. Guliyev, A.A. Ahmadli, M.N. Omarova, L. Softova, \emph{Global regularity in Orlicz-Morrey spaces of solutions to nondivergence elliptic equations with VMO coefficients}, Electron. J. Differential Equations 2018, Paper No. 110, 24 pp.

\bibitem{GulOmSoftProcIMM} V.S. Guliyev, M.N. Omarova, L. Softova,
\textit{The Dirichlet problem in a class of generalized weighted Morrey spaces}, Proc. Inst. Math. Mech. Natl. Acad. Sci. Azerb. 45 (2) (2019), 1-19.

\bibitem{HamzTrAMEA} V.H. Hamzayev, \textit{Sublinear operators with rough kernel generated by Calderon-Zygmund operators and their commutators on generalized weighted Morrey spaces}, Trans. Natl. Acad. Sci. Azerb. Ser. Phys.-Tech. Math. Sci. {\bf 38} (1) (2018), Mathematics, 79-94.

\bibitem{KomShir} Y. Komori and S. Shirai, \emph{Weighted Morrey spaces and a singular integral operator}, Math. Nachr. 282 (2) (2009), 219-231.

\bibitem{Morr1938} C.B. Morrey, \textit{On the solutions of quasi-linear elliptic partial differential equations},
Trans. Amer. Math. Soc. 43 (1) (1938), 126-166.

\bibitem{Muc}
B. Muckenhoupt, \emph{Weighted norm ineqaulities for the Hardy maximal function}, Trans. Amer. Math. Soc. 165 (1972), 207-226.

\bibitem{Miz1991} T. Mizuhara, \textit{Boundedness of some classical operators on generalized Morrey spaces},
Harmonic Anal., Proc. Conf., Sendai/Jap. 1990, ICM-90 Satell. Conf. Proc. (1991), 183-189.

\bibitem{Nakai1994} E. Nakai, \textit{Hardy-Littlewood maximal operator, singular integral operators
and the Reisz potentials on generalized Morrey spaces}, Math Nachr. 166 (1994), 95-103.

\end{thebibliography}
\end{document}